\documentclass[]{interact}

\usepackage{epstopdf}
\usepackage[caption=false]{subfig}
\usepackage{xcolor}
\usepackage{amsmath, amssymb, graphicx}
\usepackage{caption}
\usepackage{amsmath}
\usepackage{amsmath,cases}
\usepackage{mathtools}
\usepackage{xcolor}
\usepackage[numbers,sort&compress]{natbib}
\usepackage{mathtools}
\bibpunct[, ]{[}{]}{,}{n}{,}{,}
\renewcommand\bibfont{\fontsize{10}{12}\selectfont}
\makeatletter
\def\NAT@def@citea{\def\@citea{\NAT@separator}}
\makeatother

\theoremstyle{plain}
\newtheorem{theorem}{Theorem}[section]
\newtheorem{lemma}[theorem]{Lemma}
\newtheorem{corollary}[theorem]{Corollary}
\newtheorem{proposition}[theorem]{Proposition}

\theoremstyle{definition}
\newtheorem{definition}[theorem]{Definition}

\theoremstyle{remark}
\newtheorem{remark}{Remark}

\DeclareMathOperator*{\einf}{ess\ inf}

\begin{document}

\articletype{Research Paper}

\title{Nehari manifold approach for a singular multi-phase variable exponent problem}

\author{\name{Mustafa Avci\thanks{CONTACT M.~Avci. Email:  mavci@athabascau.ca (primary) \& avcixmustafa@gmail.com}}
\affil{Faculty of Science and Technology, Applied Mathematics, Athabasca University, AB, Canada}}

\maketitle

\begin{abstract}
This paper is concerned with a singular multi-phase problem with variable singularities. The main tool used is the Nehari manifold approach. Existence of at least two positive solutions with positive-negative energy levels are obtained. For the illustration of the main results, an application from fluid flow in a heterogeneous porous medium is provided.
\end{abstract}


\begin{amscode}
35A01; 35A15; 35D30; 35J66; 35J75
\end{amscode}

\section{Introduction}
In this article, we study the following multi-phase singular Dirichlet problem

\begin{equation}\label{e1.1}
\begin{cases}
\begin{array}{rlll}
-\mathcal{A}(u)+\mathcal{B}(u)&=m_1(x)|u|^{s(x)-2}u+\lambda m_2(x){u^{-\beta(x)}} \text{ in }\Omega, \\
u&>0  \text{ in }\Omega,\\
u&=0  \text{ on }\partial \Omega,\tag{$\mathcal{P_\lambda}$}
\end{array}
\end{cases}
\end{equation}
with
\begin{align*}
\mathcal{A}(u):= \mathrm{div}(|\nabla u|^{p(x)-2}\nabla u+\mu_1(x)|\nabla u|^{q(x)-2}\nabla u+\mu_2(x)|\nabla u|^{r(x)-2}\nabla u),
\end{align*}
and
\begin{align*}
\mathcal{B}(u):=|x|^{-p(x)}|u|^{p(x)-2}u+\mu_1(x)|x|^{-q(x)}|u|^{q(x)-2}u+\mu_2(x)|x|^{-r(x)}|u|^{r(x)-2}u,
\end{align*}
where $\Omega$ is a bounded domain in $\mathbb{R}^N$ $(N\geq2)$ with Lipschitz boundary; $\beta:\overline{\Omega}\rightarrow (0,1)$ is a continuous function; $m_1, m_2 \in C(\overline{\Omega })$ are
non-negative weight functions with compact support in $\Omega$; $p,q,r,s \in C_+(\overline{\Omega })$ with $1<p(x)<q(x)<r(x)<s(x)$; $0\leq \mu_1(\cdot),\mu_2(\cdot)\in L^\infty(\Omega)$; and $\lambda>0$ is a parameter.\\

The problem (\ref{e1.1}) has a strong potential to model various real-world phenomena due to its versatile terms as explained below.\\
The operator
\begin{equation}\label{e1.2a}
\mathrm{div}(|\nabla u|^{p(x)-2}\nabla u+\mu_1(x)|\nabla u|^{q(x)-2}\nabla u+\mu_2(x)|\nabla u|^{r(x)-2}\nabla u)
\end{equation}
 which governs anisotropic and heterogeneous diffusion associated with the energy functional
\begin{equation}\label{e1.2b}
u\to \int_\Omega\left(\frac{|\nabla u|^{p(x)}}{p(x)}+\mu_1(x)\frac{|\nabla u|^{q(x)}}{q(x)}+\mu_2(x)\frac{|\nabla u|^{r(x)}}{r(x)}\right)dx,\,\ u \in W_0^{1,\mathcal{T}}(\Omega)
\end{equation}
is called a "multi-phase" operator because it encapsulates three different types of elliptic behavior within the same framework. This multi-phase behavior allows the model to capture phenomena where materials or processes exhibit different characteristics in different regions—for example, materials that are harder in some areas and softer in others.\\
The terms $\frac{|u|^{p(x)-2}u}{|x|^{p(x)}}$, $\mu_1(x)\frac{|u|^{q(x)-2}u}{|x|^{q(x)}}$ and $\mu_2(x)\frac{|u|^{r(x)-2}u}{|x|^{r(x)}}$ could be seen as distance-dependent singularities at different nonlinear growths, which can model localized phenomena (e.g., near wells, defects, or critical points).\\
The term $\lambda m_2(x){u^{-\beta(x)}}$ introduces singular behavior as $u \to 0$, capturing extreme effects like strong reactions or concentrated forces.\\
The term $m_1(x)|u|^{s(x)-2}u$ describes nonlinear external force such us nonlinear fluid responses, or heat sources varying spatially depending on $m_1(x)$.\\

The functional of type (\ref{e1.2b}) was introduced in \cite{de2019regularity} for the constant exponents, where the authors obtain regularity results for multi-phase variational problems. Then, in \cite{vetro2024priori}, the author study Dirichlet problems driven by multi-phase operators with variable exponents, and provides a priori upper bounds for the weak solutions. Recently, in \cite{dai2024regularity}, the authors study the multi-phase operators with variable exponents and analyze the associated Musielak-Orlicz Sobolev spaces, extend Sobolev embedding results, and establish key regularity properties. Additionally, they prove existence and uniqueness results for Dirichlet problems with gradient-dependent nonlinearity and derive local regularity estimates.\\

We also would like to provide a historical perspective for the development of the double-phase operators associated with the energy functional
\begin{equation}\label{e1.2bc}
u\to \int_\Omega\left(\frac{|\nabla u|^{p}}{p}+\mu(x)\frac{|\nabla u|^{q}}{q}\right)dx,\,\ u \in W_0^{1,\mathcal{H}}(\Omega).
\end{equation}
The functional of type (\ref{e1.2bc}) was introduced in \cite{zhikov1987averaging}. Subsequently, many studies have been devoted in this direction (see, e.g. \cite{baroni2015harnack,baroni2018regularity,colombo2015bounded,colombo2015regularity,marcellini1991regularity,marcellini1989regularity}) due to its applicability into different disciplines.\\

Some related papers that utilized the Nehari manifold approach for constant-variable exponent double phase operators are as follows.\\
In \cite{wulong2101existence}, the authors study the following quasilinear elliptic equation
\begin{equation}\label{e1.2d}
\begin{cases}
\begin{array}{rlll}
&-\mathrm{div}(|\nabla u|^{p-2} \nabla u+\mu(x)|\nabla u|^{q-2}\nabla u)=a(x)|u|^{-\gamma}+ \lambda {u^{r-1}} \text{ in }\Omega, \\
&u=0  \text{ on }\partial \Omega,
\end{array}
\end{cases}
\end{equation}
involving the constant exponent double phase operator and a singular-parametric right-hand side. Using the fibering method and Nehari manifold, they prove the existence of at least two weak solutions for small parameters, where $a \in L^{\infty}(\overline{\Omega })$ is non-negative weight function; $1<p<N$, $1<p<q<r<p^*$; $0<\gamma<1$; $0\leq \mu(\cdot)\in L^\infty(\Omega)$; and $\lambda>0$ is a parameter.\\

In \cite{crespo2024nehari},  the authors study the following quasilinear elliptic equations with a variable exponent double phase operator and superlinear right-hand sides
\begin{equation}\label{e1.2e}
\begin{cases}
\begin{array}{rlll}
&-\mathrm{div}(|\nabla u|^{p(x)-2} \nabla u+\mu(x)|\nabla u|^{q(x)-2}\nabla u)=f(x,u) \text{ in }\Omega, \\
&u=0  \text{ on }\partial \Omega,
\end{array}
\end{cases}
\end{equation}
and they prove the existence of positive, negative, and sign-changing solutions using the Nehari manifold approach, also providing insights into the nodal domains, where $p,q \in C(\overline{\Omega })$ with $1<p(x)<q(x)<p^*(x)$, $1<p(x)<N$, and $p(\cdot)$ satisfies a monotonicity condition. We also refer the interested reader to the studies \cite{mashiyev2010nehari,saiedinezhad2015fibering,saoudi2019singular} and the references there in for broad applications of the Nehari manifold approach.\\

Motivated mainly by the problem studied in \cite{mashiyev2010nehari} and the papers mentioned above, we study a multi-phase singular Dirichlet problem on the Nehari
manifold in a newly introduced Musielak-Orlicz Sobolev space $W_0^{1,\mathcal{T}}(\Omega)$. The problem (\ref{e1.1}) contains two kinds of singularities: a distance-dependent singularity governed by a Hardy-type potential, and a term with a blow-up behavior as $u\to 0$. Inasmuch as including these two types of singularities in the same setting makes (\ref{e1.1}) more versatile in terms of its applicability of modelling different phenomena, and hence interesting, it has been also challenging to keep these singularities under the control while carrying out the necessary analysis. To the best of our knowledge, this is the first work dealing with the variable exponent multi-phase operator given in the form (\ref{e1.1}) where the Nehari manifold approach is applied. In this regard, the paper is original, and thus, it is expected to be beneficial in the related field.\\

The paper is organised as follows. In Section 2, we first provide some background for the theory of variable Sobolev spaces $W_{0}^{1,p(x)}(\Omega)$ and the Musielak-Orlicz Sobolev space $W_0^{1,\mathcal{T}}(\Omega)$, and then obtain some crucial auxiliary results. In Section 3, we set up the variational framework for problem (\ref{e1.1}) and obtain the main results; that is, we obtain two positive solutions. In Section 4, to illustrate the main results, we provide a real-world application from the fluid flow in heterogeneous porous medium, groundwater flow in heterogeneous media.

\section{Mathematical Background and Auxiliary Results}

We start with some basic concepts of variable Lebesgue-Sobolev spaces. For more details, and the proof of the following propositions, we refer the reader to \cite{cruz2013variable,diening2011lebesgue,edmunds2000sobolev,fan2001spaces,radulescu2015partial}.
\begin{equation*}
C_{+}\left( \overline{\Omega }\right) =\left\{h\in C\left( \overline{\Omega }\right) ,\text{ } h\left( x\right) >1 \text{ for all\ }x\in
\overline{\Omega }\right\} .
\end{equation*}
For $h\in C_{+}( \overline{\Omega }) $ denote
\begin{equation*}
h^{-}:=\underset{x\in \overline{\Omega }}{\min }h(x) \leq h(x) \leq h^{+}:=\underset{x\in \overline{\Omega }}{\max}h(x) <\infty .
\end{equation*}
For any $h\in C_{+}\left(\overline{\Omega}\right) $, we define \textit{the
variable exponent Lebesgue space} by
\begin{equation*}
L^{h(x)}(\Omega) =\left\{ u\mid u:\Omega\rightarrow\mathbb{R}\text{ is measurable},\int_{\Omega }|u(x)|^{h(x)}dx<\infty \right\}.
\end{equation*}
Then, $L^{h(x)}(\Omega)$ endowed with the norm
\begin{equation*}
|u|_{h(x)}=\inf \left\{ \lambda>0:\int_{\Omega }\left\vert \frac{u(x)}{\lambda }
\right\vert^{h(x)}dx\leq 1\right\} ,
\end{equation*}
becomes a Banach space.
The convex functional $\rho :L^{h(x)}(\Omega) \rightarrow\mathbb{R}$ defined by
\begin{equation*}
\rho(u) =\int_{\Omega }|u(x)|^{h(x)}dx,
\end{equation*}
is called modular on $L^{h(x)}(\Omega)$.
\begin{proposition}\label{Prop:2.2} If $u,u_{n}\in L^{h(x)}(\Omega)$, we have
\begin{itemize}
\item[$(i)$] $|u|_{h(x)}<1 ( =1;>1) \Leftrightarrow \rho(u) <1 (=1;>1);$
\item[$(ii)$] $|u|_{h(x)}>1 \implies |u|_{h(x)}^{h^{-}}\leq \rho(u) \leq |u|_{h(x)}^{h^{+}}$;\newline
$|u|_{h(x)}\leq1 \implies |u|_{h(x)}^{h^{+}}\leq \rho(u) \leq |u|_{h(x)}^{h^{-}};$
\item[$(iii)$] $\lim\limits_{n\rightarrow \infty }|u_{n}-u|_{h(x)}=0\Leftrightarrow \lim\limits_{n\rightarrow \infty }\rho (u_{n}-u)=0 \Leftrightarrow \lim\limits_{n\rightarrow \infty }\rho (u_{n})=\rho(u)$.
\end{itemize}
\end{proposition}
\begin{proposition}\label{Prop:2.2bb}
Let $h_1(x)$ and $h_2(x)$ be measurable functions such that $h_1\in L^{\infty}(\Omega )$ and $1\leq h_1(x)h_2(x)\leq \infty$ for a.e. $x\in \Omega$. Let $u\in L^{h_2(x)}(\Omega ),~u\neq 0$. Then
\begin{itemize}
\item[$(i)$] $\left\vert u\right\vert _{h_1(x)h_2(x)}\leq 1\text{\ }\Longrightarrow
\left\vert u\right\vert_{h_1(x)h_2(x)}^{h_1^{+}}\leq \left\vert \left\vert
u\right\vert ^{h_1(x)}\right\vert_{h_2(x)}\leq \left\vert
u\right\vert _{h_1(x)h_2(x)}^{h_1^{-}}$
\item[$(ii)$] $\left\vert u\right\vert_{h_1(x)h_2(x)}\geq 1\ \Longrightarrow \left\vert
u\right\vert_{h_1(x)h_2(x)}^{h_1^{-}}\leq \left\vert \left\vert u\right\vert^{h_1(x)}\right\vert_{h_2(x)}\leq \left\vert u\right\vert_{h_1(x)h_2(x)}^{h_1^{+}}$
\item[$(iii)$] In particular, if $h_1(x)=h$ is constant then
\begin{equation*}
\left\vert \left\vert u\right\vert^{h}\right\vert_{h_2(x)}=\left\vert u\right\vert _{hh_2(x)}^{h}.
\end{equation*}
\end{itemize}
\end{proposition}

The variable exponent Sobolev space $W^{1,h(x)}( \Omega)$ is defined by
\begin{equation*}
W^{1,h(x)}( \Omega) =\{u\in L^{p(x) }(\Omega) : |\nabla u| \in L^{h(x)}(\Omega)\},
\end{equation*}
with the norm
\begin{equation*}
\|u\|_{1,h(x)}:=|u|_{h(x)}+|\nabla u|_{h(x)},
\end{equation*}
for all $u\in W^{1,h(x)}(\Omega)$.\\
\begin{proposition}\label{Prop:2.4} If $1<h^{-}\leq h^{+}<\infty $, then the spaces
$L^{h(x)}(\Omega)$ and $W^{1,h(x)}(\Omega)$ are separable and reflexive Banach spaces.
\end{proposition}
The space $W_{0}^{1,h(x)}(\Omega)$ is defined as
$\overline{C_{0}^{\infty }(\Omega )}^{\|\cdot\|_{1,h(x)}}=W_{0}^{1,h(x)}(\Omega)$, and hence, it is the smallest closed set that contains $C_{0}^{\infty }(\Omega )$. Therefore, $W_{0}^{1,h(x)}(\Omega)$ is also a separable and reflexive Banach space due to the inclusion $W_{0}^{1,h(x)}(\Omega) \subset W^{1,h(x)}(\Omega)$. \\
Note that as  a consequence of Poincar\'{e} inequality, $\|u\|_{1,h(x)}$ and $|\nabla u|_{h(x)}$ are equivalent
norms on $W_{0}^{1,h(x)}(\Omega)$. Therefore, for any $u\in W_{0}^{1,h(x)}(\Omega)$ we can define an equivalent norm $\|u\|$ such that
\begin{equation*}
\|u\| :=|\nabla u|_{h(x)}.
\end{equation*}
\begin{proposition}\label{Prop:2.5} Let $m\in C(\overline{\Omega })$. If $1\leq m(x)<h^{\ast }(x)$ for all $x\in
\overline{\Omega}$, then the embeddings $W^{1,h(x)}(\Omega) \hookrightarrow L^{m(x)}(\Omega)$ and $W_0^{1,h(x)}(\Omega) \hookrightarrow L^{m(x)}(\Omega)$  are compact and continuous, where
$h^{\ast}(x) =\left\{\begin{array}{cc}
\frac{Nh(x) }{N-h(x)} & \text{if }h(x)<N, \\
+\infty & \text{if }h(x) \geq N.
\end{array}
\right. $
\end{proposition}
In the sequel, we introduce the multi-phase operator, the Musielak–Orlicz space, and the Musielak–Orlicz Sobolev space, respectively.\\
We make the following assumptions.
\begin{itemize}
\item[$(H_1)$] $p,q,r,s\in C_+(\overline{\Omega})$ with $p(x)<N$; $p(x)<q(x)<r(x)<s(x)<p^*(x)$; $s^+<p^*(x)$.
\item[$(H_2)$] $\mu_1(\cdot),\mu_2(\cdot)\in L^\infty(\Omega)$ such that $\mu_1(x)\geq 0$ and $\mu_2(x)\geq 0$ for all $x\in
\overline{\Omega }$.
\end{itemize}
Under the assumptions $(H_1)$ and $(H_2)$, we define the nonlinear function $\mathcal{T}:\Omega\times [0,\infty]\to [0,\infty]$, i.e. the \textit{multi-phase operator}, by
\[
\mathcal{T}(x,t)=t^{p(x)}+\mu_1(x)t^{q(x)}+\mu_2(x)t^{r(x)}\ \text{for all}\ (x,t)\in \Omega\times [0,\infty].
\]
Then the corresponding modular $\rho_\mathcal{T}(\cdot)$ is given by
\[
\displaystyle\rho_\mathcal{T}(u):=\int_\Omega\mathcal{T}(x,|u|)dx=
\int_\Omega\left(|u(x)|^{p(x)}+\mu_1(x)|u(x)|^{q(x)}+\mu_2(x)|u(x)|^{r(x)}\right)dx.
\]
The \textit{Musielak-Orlicz space} $L^{\mathcal{T}}(\Omega)$, is defined by
\[
L^{\mathcal{T}}(\Omega)=\left\{u:\Omega\to \mathbb{R}\,\, \text{measurable};\,\, \rho_{\mathcal{T}}(u)<+\infty\right\},
\]
endowed with the Luxemburg norm
\[
\|u\|_{\mathcal{T}}:=\inf\left\{\zeta>0: \rho_{\mathcal{T}}\left(\frac{u}{\zeta}\right)\leq 1\right\}.
\]
Analogous to Proposition \ref{Prop:2.2}, there are similar relationship between the modular $\rho_{\mathcal{T}}(\cdot)$ and the norm $\|\cdot\|_{\mathcal{T}}$, see \cite[Proposition 3.2]{dai2024regularity} for a detailed proof.

\begin{proposition}\label{Prop:2.2a}
Assume $(H_1)$ hold, and $u\in L^{\mathcal{H}}(\Omega)$. Then
\begin{itemize}
\item[$(i)$] If $u\neq 0$, then $\|u\|_{\mathcal{T}}=\zeta\Leftrightarrow \rho_{\mathcal{T}}(\frac{u}{\zeta})=1$,
\item[$(ii)$] $\|u\|_{\mathcal{T}}<1\ (\text{resp.}\ >1, =1)\Leftrightarrow \rho_{\mathcal{T}}(\frac{u}{\zeta})<1\ (\text{resp.}\ >1, =1)$,
\item[$(iii)$] If $\|u\|_{\mathcal{T}}<1\Rightarrow \|u\|_{\mathcal{T}}^{r^+}\leq \rho_{\mathcal{T}}(u)\leq \|u\|_{\mathcal{T}}^{p^-}$,
\item[$(iv)$]If $\|u\|_{\mathcal{T}}>1\Rightarrow \|u\|_{\mathcal{T}}^{p^-}\leq \rho_{\mathcal{T}}(u)\leq \|u\|_{\mathcal{T}}^{r^+}$,
\item[$(v)$] $\|u\|_{\mathcal{T}}\to 0\Leftrightarrow \rho_{\mathcal{T}}(u)\to 0$,
\item[$(vi)$]$\|u\|_{\mathcal{T}}\to +\infty\Leftrightarrow \rho_{\mathcal{T}}(u)\to +\infty$,
\item[$(vii)$] $\|u\|_{\mathcal{T}}\to 1\Leftrightarrow \rho_{\mathcal{T}}(u)\to 1$,
\item[$(viii)$] If $u_n\to u$ in $L^{\mathcal{T}}(\Omega)$, then $\rho_{\mathcal{T}}(u_n)\to\rho_{\mathcal{T}}(u)$.
\end{itemize}
\end{proposition}
The \textit{Musielak-Orlicz Sobolev space} $W^{1,\mathcal{T}}(\Omega)$ is defined by
\[
W^{1,\mathcal{T}}(\Omega)=\left\{u\in L^{\mathcal{T}}(\Omega):
|\nabla u|\in L^{\mathcal{H}}(\Omega)\right\},
\]
and equipped with the norm
\[
\|u\|_{1,\mathcal{T}}:=\|\nabla u\|_{\mathcal{T}}+\|u\|_{\mathcal{T}},
\]
where $\|\nabla u\|_{\mathcal{T}}=\|\,|\nabla u|\,\|_{\mathcal{T}}$.\\
The space $W_0^{1,\mathcal{T}}(\Omega)$ is defined as $\overline{C_{0}^{\infty }(\Omega )}^{\|\cdot\|_{1,\mathcal{T}}}=W_0^{1,\mathcal{T}}(\Omega)$. Notice that $L^{\mathcal{T}}(\Omega), W^{1,\mathcal{T}}(\Omega)$ and $W_0^{1,\mathcal{T}}(\Omega)$ are uniformly convex and reflexive Banach spaces, and the following embeddings hold \cite[Propositions 3.1, 3.3]{dai2024regularity}.
\begin{proposition}\label{Prop:2.7a}
Let $(H_1)$ and $(H_2)$ be satisfied. Then the following embeddings hold:
\begin{itemize}
\item[$(i)$] $L^{\mathcal{T}}(\Omega)\hookrightarrow L^{h(\cdot)}(\Omega), W^{1,\mathcal{T}}(\Omega)\hookrightarrow W^{1,h(\cdot)}(\Omega)$, $W_0^{1,\mathcal{T}}(\Omega)\hookrightarrow W_0^{1,h(\cdot)}(\Omega)$ are continuous for all $h\in C(\overline{\Omega})$ with $1\leq h(x)\leq p(x)$ for all $x\in \overline{\Omega}$.
\item[$(ii)$] $W^{1,\mathcal{T}}(\Omega)\hookrightarrow L^{h(\cdot)}(\Omega)$ and $W_0^{1,\mathcal{T}}(\Omega)\hookrightarrow L^{h(\cdot)}(\Omega)$ are compact for all $h\in C(\overline{\Omega})$ with $1\leq h(x)< p^*(x)$ for all $x\in \overline{\Omega}$.
\end{itemize}
\end{proposition}
As a conclusion of Proposition \ref{Prop:2.7a}:\\
We have the continuous embedding $W_0^{1,\mathcal{T}}(\Omega)\hookrightarrow L^{h(\cdot)}(\Omega)$, and there is a constant $c_{\mathcal{T}}$ such that
\[
\|u\|_{h(\cdot)}\leq c_{\mathcal{T}}\|u\|_{1,\mathcal{T},0}.
\]
As well, $W_0^{1,\mathcal{H}}(\Omega)$ is compactly embedded in $L^{\mathcal{H}}(\Omega)$.\\
Thus,
$W_0^{1,\mathcal{T}}(\Omega)$ can be equipped with the equivalent norm
\[
\|u\|_{1,\mathcal{T},0}:=\|\nabla u\|_{\mathcal{T}}.
\]
\begin{proposition}\label{Prop:2.7}
For the convex functional
$$
\varrho_{\mathcal{T}}(u):=\int_\Omega\left(\frac{|\nabla u|^{p(x)}}{p(x)}+\mu_1(x)\frac{|\nabla u|^{q(x)}}{q(x)}+\mu_2(x)\frac{|\nabla u|^{r(x)}}{r(x)}\right)dx,
$$
we have $\varrho_{\mathcal{T}} \in C^{1}(W_0^{1,\mathcal{T}}(\Omega),\mathbb{R})$ with the derivative
$$
\langle\varrho^{\prime}_{\mathcal{T}}(u),\varphi\rangle=\int_{\Omega}(|\nabla u|^{p(x)-2}\nabla u+\mu_1(x)|\nabla u|^{q(x)-2}\nabla u+\mu_2(x)|\nabla u|^{r(x)-2}\nabla u)\cdot\nabla \varphi dx,
$$
for all  $u, \varphi \in W_0^{1,\mathcal{T}}(\Omega)$, where $\langle \cdot, \cdot\rangle$ is the dual pairing between $W_0^{1,\mathcal{T}}(\Omega)$ and its dual $W_0^{1,\mathcal{T}}(\Omega)^{*}$ \cite{dai2024regularity}.
\end{proposition}
\begin{remark}\label{Rem:3.4ab}
Notice that by Propositions \ref{Prop:2.2a} and the equivalency of the norms $\|u\|_{1,\mathcal{T},0}$ and $\|\nabla u\|_{\mathcal{T}}$, we have the relations:
\begin{equation}\label{e2.de}
\frac{1}{r^+}\|u\|^{p^-}_{1,\mathcal{T},0}\leq \frac{1}{r^+}\rho_{\mathcal{T}}(\nabla u)\leq \varrho_{\mathcal{T}}(u) \leq \frac{1}{p^-}\rho_{\mathcal{T}}(\nabla u)\leq \frac{1}{p^-}\|u\|^{r^+}_{1,\mathcal{T},0} \text{ if } \|u\|_{1,\mathcal{T},0}>1,
\end{equation}
\begin{equation}\label{e2.df}
\frac{1}{r^+}\|u\|^{r^+}_{1,\mathcal{T},0}\leq \frac{1}{r^+}\rho_{\mathcal{T}}(\nabla u)\leq \varrho_{\mathcal{T}}(u) \leq \frac{1}{p^-}\rho_{\mathcal{T}}(\nabla u)\leq \frac{1}{p^-}\|u\|^{p^-}_{1,\mathcal{T},0} \text{ if } \|u\|_{1,\mathcal{T},0}\leq 1.
\end{equation}
\end{remark}

Throughout the paper, we assume the following.
\begin{itemize}
\item[$(A_{\beta})$] $\beta:\overline{\Omega}\to (0,1)$ is a continuous function such that $0<\beta^-\leq \beta(x)\leq \beta^+<1$, that is
\begin{equation*}
0<1-\beta^+\leq 1- \beta(x)\leq 1-\beta^-<1.
\end{equation*}
\item[$(A_{\alpha})$] $\alpha \in C_+(\overline{\Omega })$ fulfilling
\begin{equation*}
1\leq \alpha_0(x)(1-\beta^+)\leq\alpha_0(x)(1-\beta(x))\leq \alpha_0(x)(1-\beta^-)<p^*(x),\,\, \forall x \in \overline{\Omega},
\end{equation*}
where $\alpha_0^-\leq \alpha_0(x)\leq \alpha_0^+$ such that $\alpha_0(x)=\frac{\alpha(x)}{\alpha(x)-1}$.
\item[$(A_{\gamma})$] $\gamma \in C_+(\overline{\Omega })$ satisfying
\begin{equation*}
1\leq s(x)\gamma_0(x)<p^*(x),\,\, \forall x \in \overline{\Omega},
\end{equation*}
where $\gamma_0^-\leq \gamma_0(x)\leq \gamma_0^+$ such that $\gamma_0(x)=\frac{\gamma(x)}{\gamma(x)-1}$.
\end{itemize}
In the sequel, we obtain some auxiliary results which are necessary to obtain the main results.\\
The symbol $c$ without index stands for a generic positive constant which varies from line to line.\\
\begin{lemma}\label{Lem:3.3a}
Let $m_2 \in L^{\alpha(x)}(\Omega)$. Then the embedding $W_0^{1,\mathcal{T}}(\Omega)\hookrightarrow L_{m_2(x)}^{1-\beta(x)}(\Omega)$ is compact, and the inequality
\begin{equation}\label{e3.12a}
\int_{\Omega}m_2(x)|u|^{1-\beta(x)}dx\leq c(\|u\|^{1-\beta^-}_{1,\mathcal{T},0}+\|u\|^{1-\beta^+}_{1,\mathcal{T},0})\leq c\delta_1^{1-\beta^*}
\end{equation}
holds, where
\begin{equation*}
\begin{aligned}
\delta_1^{1-\beta^*}:=
\left\{ \begin{array}{ll}
\|u\|^{1-\beta^+}_{1,\mathcal{T},0}, & \|u\|_{1,\mathcal{T},0} <1 , \\
\|u\|^{1-\beta^-}_{1,\mathcal{T},0}, & \|u\|_{1,\mathcal{T},0} \geq 1.
\end{array}\right.
\end{aligned}
\end{equation*}
\end{lemma}
\begin{proof} We refer the reader to the papers \cite{fan2005solutions,mashiyev2010nehari}
for the analogous results obtained in the space $W_{0}^{1,p(x)}(\Omega)$.\\
Step I: \\
Let's define $l(x):=\alpha_0(x)(1-\beta(x))$. Then $W_0^{1,\mathcal{T}}(\Omega)\hookrightarrow L^{l(x)}(\Omega)$. Thus, for $u \in W_0^{1,\mathcal{T}}(\Omega)$, $\int_{\Omega}||u|^{1-\beta(x)}|^{\alpha_0(x)}dx<\infty$, that is $ |u|^{1-\beta(x)} \in L^{\alpha_0(x)}(\Omega)$.  Using H\"{o}lder inequality, we obtain
\begin{equation}\label{e3.12abb}
\int_{\Omega}m_2(x)|u_n|^{1-\beta(x)}dx\leq c|m_2|_{\alpha(x)}||u_n|^{1-\beta(x)}|_{\alpha_0(x)}<\infty,
\end{equation}
which means $W_0^{1,\mathcal{T}}(\Omega)\subset L_{m_2(x)}^{1-\beta(x)}(\Omega)$.\\
Step II:\\
Let $(u_{n}) \subset W_0^{1,\mathcal{T}}(\Omega)$ such that $u_n \rightharpoonup 0$ (weakly) in $W_0^{1,\mathcal{T}}(\Omega)$. Then $u_n \to 0$ (strongly) in $L^{l(x)}(\Omega)$. Using H\"{o}lder inequality, we obtain
\begin{equation}\label{e3.12ac}
\int_{\Omega}m_2(x)|u_n|^{1-\beta(x)}dx\leq c|m_2|_{\alpha(x)}||u_n|^{1-\beta(x)}|_{\alpha_0(x)} \to 0,
\end{equation}
that is $|u_n|_{(m_2(x),1-\beta(x))} \to 0$. Thus, $W_0^{1,\mathcal{T}}(\Omega)\hookrightarrow L_{m_2(x)}^{1-\beta(x)}(\Omega)$ is compact.\\
Step III: \\
From $(A_{\beta})$, we can write $\int_{\Omega}m_2(x)|u|^{1-\beta(x)}dx\leq \int_{\Omega}m_2(x)\left(|u|^{1-\beta^+}+|u|^{1-\beta^-}\right)dx$. Then using (\ref{e3.12ac}) and H\"{o}lder inequality gives
\begin{equation}\label{e3.12ad}
\int_{\Omega}m_2(x)|u_n|^{1-\beta(x)}dx\leq c|m_2|_{\alpha(x)}||u|^{1-\beta(x)}|_{\alpha_0(x)}\leq c |m_2|_{\alpha(x)}|u|^{1-\beta^-}_{\alpha_0(x)(1-\beta^-)}\leq c \|u\|^{1-\beta^-}_{1,\mathcal{T},0}.
\end{equation}
With the same logic, it reads
\begin{equation}\label{e3.12ae}
\int_{\Omega}m_2(x)|u_n|^{1-\beta(x)}dx\leq c|m_2|_{\alpha(x)}||u|^{1-\beta(x)}|_{\alpha_0(x)}\leq c |m_2|_{\alpha(x)}|u|^{1-\beta^+}_{\alpha_0(x)(1-\beta^+)}\leq c \|u\|^{1-\beta^+}_{1,\mathcal{T},0}.
\end{equation}
Thus, by (\ref{e3.12ad}) and (\ref{e3.12ae}), (\ref{e3.12ac}) follows.
\end{proof}
\begin{lemma}\label{Lem:3.3b}
Let $m_1 \in L^{\gamma(x)}(\Omega)$. Then the embedding $W_0^{1,\mathcal{T}}(\Omega)\hookrightarrow L_{m_1(x)}^{s(x)}(\Omega)$ is compact, and the inequality
\begin{equation}\label{e3.12a}
\int_{\Omega}m_1(x)|u|^{s(x)}dx\leq c(\|u\|^{s^-}_{1,\mathcal{T},0}+\|u\|^{s^+}_{1,\mathcal{T},0})\leq c\delta_2^{s^*}
\end{equation}
holds, where
\begin{equation*}
\begin{aligned}
\delta_2^{s^*}:=
\left\{ \begin{array}{ll}
\|u\|^{s^-}_{1,\mathcal{T},0}, & \|u\|_{1,\mathcal{T},0} <1 , \\
\|u\|^{s^+}_{1,\mathcal{T},0}, & \|u\|_{1,\mathcal{T},0} \geq 1.
\end{array}\right.
\end{aligned}
\end{equation*}
\end{lemma}
\begin{proof}
If one carries out a similar analysis as with the Lemma \ref{Lem:3.3a}, the proof follows.
\end{proof}

\begin{lemma}\label{Lem:3.3c}[Hardy-type Inequality] Assume that $p,q,r$ satisfies $(H_1)$.\\
$(i)$ There exists a positive constant $C_{N}(p,r)$ such that for all $u \in W_0^{1,\mathcal{T}}(\Omega)$ the inequality
\begin{align}\label{e3.11mm}
\int_{\Omega}\left(\frac{|u|^{p(x)}}{p(x)|x|^{p(x)}}+\mu_1(x)\frac{|u|^{q(x)}}{q(x)|x|^{q(x)}}+\mu_2(x)\frac{|u|^{r(x)}}{r(x)|x|^{r(x)}}\right)dx \leq C_{N}(p,r)\|u\|_{1,\mathcal{T},0}^{\varphi_0},
\end{align}
holds, where $C_{N}(p,r):=(p^-)^{-1}\hat{c}_M(1+\|\mu_1\|_{\infty}+\|\mu_2\|_{\infty})\max\{C_{N}(r^+),C_{N}(p^-)\}$ with $C_{N}(r^+):=\left(\frac{r^+}{(N-2)(r^+-1)}\right)^{r^+}$, $C_{N}(p^-):=\left(\frac{p^-}{(N-2)(p^--1)}\right)^{p^-}$; $\hat{c}_M:=\max\{\hat{c}_1,\hat{c}_2\}$ with $\hat{c}_1>0,\hat{c}_2>0$ are the embedding constants not depending on $u$; and $\varphi_0=p^-$ if $\|u\|_{1,\mathcal{T},0}<1$, $\varphi_0=r^{+}$ if $\|u\|_{1,\mathcal{T},0}\geq1$.\\
$(ii)$ There exists $0<x_{*} \in \mathbb{R}$ such that the inequality
\begin{equation}\label{e3.11mn}
\int_{\Omega}\left(\frac{|u|^{p(x)}}{|x|^{p(x)}}+\mu_1(x)\frac{|u|^{q(x)}}{|x|^{q(x)}}+\mu_2(x)\frac{|u|^{r(x)}}{|x|^{r(x)}}\right)dx \geq \frac{1}{x_{*}^\tau}\|u\|^{\tau}_{\mathcal{T}},\,\,\ \forall u \in W_0^{1,\mathcal{T}}(\Omega)
\end{equation}
holds, where $\tau\in \{p^-, r^+\}$. Notice that (\ref{e3.11mn}) implies the real number $\frac{1}{x_{*}^\tau}\|u\|^{\tau}_{\mathcal{T}}$ is a positive lower bound for $\int_{\Omega}\left(\frac{|u|^{p(x)}}{|x|^{p(x)}}+\mu_1(x)\frac{|u|^{q(x)}}{|x|^{q(x)}}+\mu_2(x)\frac{|u|^{r(x)}}{|x|^{r(x)}}\right)dx$.
\end{lemma}

\begin{proof} $(i)$ The proof is split in two cases.\\
Case I: $|x| \leq |u(x)|$.\\
Using $(H_1)$, for all $x \in \overline{\Omega}$, $x\neq 0$, we have
\begin{align}\label{e3.11nn}
&\frac{|u|^{p(x)}}{p(x)|x|^{p(x)}}+\mu_1(x)\frac{|u|^{q(x)}}{q(x)|x|^{q(x)}}+\mu_2(x)\frac{|u|^{r(x)}}{r(x)|x|^{r(x)}} \nonumber \\
&\leq \frac{1}{p^-}\left(\frac{|u(x)|^{p(x)}}{|x|^{p(x)}}+\mu_1(x)\frac{|u(x)|^{q(x)}}{|x|^{q(x)}}+\mu_2(x)\frac{|u(x)|^{r(x)}}{|x|^{r(x)}}\right)\nonumber \\
&\leq \frac{(1+\|\mu_1\|_{\infty}+\|\mu_2\|_{\infty})}{p^-} \frac{|u(x)|^{r^+}}{|x|^{r^+}}.
\end{align}
We can write (see Theorem 7 in \cite{evans2022partial})
\begin{align}\label{e3.11na}
\frac{|u(x)|^{r^{+}}}{|x|^{r^{+}}}=\frac{1}{(1-r^{+})} |u|^{r^{+}} \frac{x}{|x|} \cdot \nabla\left(\frac{1}{|x|^{r^{+}-1}}\right).
\end{align}
Then, applying the divergence theorem gives
\begin{align}\label{e3.11nc}
\int_{\Omega}\frac{|u|^{r^{+}}}{|x|^{r^{+}}}dx =\frac{1}{(r^{+}-1)} \int_{\Omega}\nabla \cdot \left(|u|^{r^{+}} \frac{x}{|x|}\right) \frac{1}{|x|^{r^{+}-1}} dx+\frac{1}{(1-r^{+})}\int_{\partial \Omega}|u|^{r^{+}} \frac{x}{|x|^{r^{+}}} \cdot \textbf{n}dS,
\end{align}
which is equal to
\begin{align}\label{e3.11ng}
(2-N)\int_{\Omega}\frac{|u|^{r^{+}}}{|x|^{r^{+}}}dx =\frac{r^{+}}{(r^{+}-1)} \int_{\Omega}|u|^{r^{+}-2}u \nabla u \cdot \frac{x}{|x|^{r^{+}}}dx.
\end{align}
Then,
\begin{align}\label{e3.11nk}
(N-2)\int_{\Omega}\frac{|u|^{r^{+}}}{|x|^{r^{+}}}dx\leq \frac{r^{+}}{(r^{+}-1)} \int_{\Omega}\frac{|u|^{r^{+}-1}}{|x|^{r^{+}-1}} |\nabla u| dx.
\end{align}
Using the H\"{o}lder inequality, we obtain
\begin{align}\label{e3.11np}
\int_{\Omega}\frac{|u|^{r^{+}}}{|x|^{r^{+}}}dx\leq C_{N}(r^{+})\int_{\Omega}|\nabla u|^{r^{+}}dx.
\end{align}
By the continuous embedding $W_0^{1,\mathcal{T}}(\Omega)\hookrightarrow W_0^{1,r^+}(\Omega)$, we get
\begin{align}\label{e3.11nr}
&\int_{\Omega}\left(\frac{|u|^{p(x)}}{p(x)|x|^{p(x)}}+\mu_1(x)\frac{|u|^{q(x)}}{q(x)|x|^{q(x)}}+\mu_2(x)\frac{|u|^{r(x)}}{r(x)|x|^{r(x)}}\right)dx \nonumber \\
&\leq \frac{\hat{c}_1(1+\|\mu_1\|_{\infty}+\|\mu_2\|_{\infty})C_{N}(q^{+})}{p^-}\|u\|_{1,\mathcal{T},0}^{q^{+}}.
\end{align}
Case II: $|x| >|u(x)|$.\\
Considering $(H_1)$ again, for all $x \in \overline{\Omega}$, $x\neq 0$, it reads
\begin{equation}\label{e3.11nnm}
\frac{|u|^{p(x)}}{p(x)|x|^{p(x)}}+\mu_1(x)\frac{|u|^{q(x)}}{q(x)|x|^{q(x)}}+\mu_2(x)\frac{|u|^{r(x)}}{r(x)|x|^{r(x)}}\leq \frac{(1+\|\mu_1\|_{\infty}+\|\mu_2\|_{\infty})}{p^-} \frac{|u(x)|^{p^-}}{|x|^{p^-}}.
\end{equation}
Carrying out a similar analysis gives
\begin{align}\label{e3.11ns}
&\int_{\Omega}\left(\frac{|u|^{p(x)}}{p(x)|x|^{p(x)}}+\mu_1(x)\frac{|u|^{q(x)}}{q(x)|x|^{q(x)}}+\mu_2(x)\frac{|u|^{r(x)}}{r(x)|x|^{r(x)}}\right)dx\nonumber \\
& \leq \frac{\hat{c}_2(1+\|\mu_1\|_{\infty}+\|\mu_2\|_{\infty})}{p^-} \|u\|_{1,\mathcal{T},0}^{p^-}.
\end{align}
Hence, from (\ref{e3.11nr}) and (\ref{e3.11ns}), we conclude that for all $u \in W_0^{1,\mathcal{T}}(\Omega)$ it holds
\begin{equation*}
\int_{\Omega}\left(\frac{|u|^{p(x)}}{|x|^{p(x)}}+\mu_1(x)\frac{|u|^{q(x)}}{|x|^{q(x)}}+\mu_2(x)\frac{|u|^{r(x)}}{|x|^{r(x)}}\right)dx \leq C_{N}(p,r)\|u\|_{1,\mathcal{T},0}^{\varphi_0}.
\end{equation*}
$(ii)$ For $0\neq x \in \mathbb{R}^N$, we have $0<|x| \leq N\times\max_{x_i\in \mathbb{R}^N}|x_i|:=x_{*} \in \mathbb{R}$. Thus, $\max\{|x|^{p(x)},|x|^{q(x)},|x|^{r(x)}\}\leq x_{*}^\tau$, where $\tau\in \{p^-, r^+\} $.
Therefore,
\begin{align}\label{e3.11nt}
\int_{\Omega}\left(\frac{|u|^{p(x)}}{|x|^{p(x)}}+\mu_1(x)\frac{|u|^{q(x)}}{|x|^{q(x)}}+\mu_2(x)\frac{|u|^{r(x)}}{|x|^{r(x)}}\right)dx & \geq \frac{1}{x_{*}^\tau}\int_{\Omega}\left(|u|^{p(x)}+\mu_1(x)|u|^{q(x)}+\mu_2(x)|u|^{r(x)}\right)dx\nonumber \\
&\geq \frac{1}{x_{*}^\tau}\|u\|^{\tau}_{\mathcal{T}}.
\end{align}
\end{proof}
In the sequel, we let $\mathcal{F}(u):=\int_{\Omega}\left(\frac{|u|^{p(x)}}{p(x)|x|^{p(x)}}+\mu_1(x)\frac{|u|^{q(x)}}{q(x)|x|^{q(x)}}+\mu_2(x)\frac{|u|^{r(x)}}{r(x)|x|^{r(x)}}\right)dx $.
\section{Variational Framework and Main Results}
The energy functional $\mathcal{J}_{\lambda}:W_0^{1,\mathcal{T}}(\Omega)\rightarrow \mathbb{R}$ corresponding to equation (\ref{e1.1}) is
\begin{align*}
\mathcal{J}_{\lambda}(u)&=\varrho_{\mathcal{T}}(u)+\mathcal{F}(u)-\int_{\Omega}\frac{m_1(x)|u|^{s(x)}}{s(x)}dx-\lambda\int_{\Omega}\frac{m_2(x)|u|^{1-\beta(x)}}{1-\beta(x)}dx.
\end{align*}
\begin{definition}\label{Def:3.1} A function $u$ is called a weak solution to problem (\ref{e1.1}) if $u\in W_0^{1,\mathcal{T}}(\Omega)$ such that $\einf_{Q}u>0$ for any domain $Q\Subset\Omega$ it holds
\begin{align}\label{e3.2}
\langle\varrho^{\prime}_{\mathcal{T}}(u),\varphi\rangle+\langle \mathcal{F}^{\prime}(u), \varphi\rangle
=\int_{\Omega}m_1(x)|u|^{s(x)-1}\varphi dx+\lambda\int_{\Omega}m_2(x)u^{-\beta(x)}\varphi dx,
\end{align}
for all $\varphi\in W_0^{1,\mathcal{T}}(\Omega)$.
\end{definition}
Due to assumption $(H_1)$, $\mathcal{J}_{\lambda}$ is not bounded below on the whole space $W_0^{1,\mathcal{H}}(\Omega)$. Thus, we need an appropriate subspace of $W_0^{1,\mathcal{T}}(\Omega)$ where $\mathcal{J}_{\lambda}$ is bounded below, and hence, we can search for critical points of $\mathcal{J}_{\lambda}$ which corresponds to the weak solutions to problem \ref{e1.1}. The best candidate for such a subspace is the well-known \textit{Nehari manifold} defined as follows:
\begin{equation*}
\mathcal{M}=\biggl\{u\in W_0^{1,\mathcal{T}}(\Omega)\backslash \{0\}: \langle \mathcal{J}_{\lambda}^{\prime}(u),u\rangle = 0 \biggl\}.
\end{equation*}
With this definition, it is clear that $\mathcal{M}$ must include all \textit{nontrivial} critical points (local minimizers) of $\mathcal{J}_{\lambda}$, if any exists. Therefore, $u \in \mathcal{M}$ if and only if
\begin{align}\label{e3.2a}
J(u):=\langle \mathcal{J}_{\lambda}^{\prime}(u),u\rangle&=\langle\varrho^{\prime}_{\mathcal{T}}(u),u\rangle+\langle \mathcal{F}^{\prime}(u), u\rangle -\int_{\Omega}m_1(x)|u|^{s(x)}dx\nonumber\\
&-\lambda\int_{\Omega}m_2(x)|u|^{1-\beta(x)} dx=0.
\end{align}
Then
\begin{align}\label{e3.2bc}
\langle J^{\prime}(u),u\rangle &= \langle \varrho^{\prime\prime}_{\mathcal{T}}(u),u\rangle +\langle \mathcal{F}^{\prime\prime}(u), u\rangle-\int_{\Omega}m_1(x)s(x)|u|^{s(x)}dx\nonumber\\
&-\lambda\int_{\Omega}m_2(x)(1-\beta(x))|u|^{1-\beta(x)} dx=0.
\end{align}
Thus, considering (\ref{e3.2a}) and (\ref{e3.2bc}) together, one can partition $\mathcal{M}$ into following three disjoint sets:
\begin{align}\label{e3.6}
&\mathcal{M}^+=\biggl\{u \in W_0^{1,\mathcal{T}}(\Omega)\setminus\{0\}: \langle J^{\prime}(u),u\rangle > 0  \biggl\},\nonumber\\
&\mathcal{M}^0=\biggl\{u \in W_0^{1,\mathcal{T}}(\Omega)\setminus\{0\}: \langle J^{\prime}(u),u\rangle=0  \biggl\},\nonumber\\
&\mathcal{M}^-=\biggl\{u \in W_0^{1,\mathcal{T}}(\Omega)\setminus\{0\}: \langle J^{\prime}(u),u\rangle <0  \biggl\}.
\end{align}

\begin{remark}\label{Rem:3.1} Note that we have the following relations
\begin{equation}\label{e3.33a}
0<\langle\varrho^{\prime}_{\mathcal{T}}(u),u\rangle=\rho_{\mathcal{T}}(\nabla u),
\end{equation}
\begin{equation}\label{e3.33b}
0<\langle\varrho^{\prime\prime}_{\mathcal{T}}(u),u\rangle=\langle \rho^{\prime}_{\mathcal{T}}(\nabla u), \nabla u\rangle,
\end{equation}
\begin{equation}\label{e3.33c}
0<p^- \rho_{\mathcal{T}}(\nabla u)\leq \langle \rho^{\prime}_{\mathcal{T}}(\nabla u), \nabla u\rangle \leq r^+ \rho_{\mathcal{T}}(\nabla u),
\end{equation}
\begin{equation}\label{e3.33d}
0<\frac{1}{r^+} \langle \mathcal{F}^{\prime}(u), u\rangle \leq \mathcal{F}(u) \leq \frac{1}{p^-} \langle \mathcal{F}^{\prime}(u), u\rangle,
\end{equation}
\begin{equation}\label{e3.33e}
0<p^- \langle \mathcal{F}^{\prime}(u), u\rangle \leq \langle\mathcal{F}^{\prime\prime}(u), u\rangle \leq r^+ \langle \mathcal{F}^{\prime}(u), u\rangle,
\end{equation}
\end{remark}

\begin{lemma}\label{Lem:3.3bc}
The functional $\mathcal{J}_{\lambda}$ is well-defined, coercive and bounded below on  $\mathcal{M}$.
\end{lemma}
\begin{proof}
Let $u \in \mathcal{M}$. Using Lemmas \ref{Lem:3.3a}-\ref{Lem:3.3c}, we obtain
\begin{equation*}
  |\mathcal{J}_{\lambda}(u)| \leq c_0 \left( \|u\|_{1,\mathcal{T},0}^{r^{+}}+\|u\|_{1,\mathcal{T},0}^{s^{+}}+\|u\|_{1,\mathcal{T},0}^{1-\beta^{-}}\right)<\infty.
\end{equation*}
Hence $\mathcal{J}_{\lambda}$ is well-defined, where $c_0:=\max\left(\frac{1}{p^-}, C_{N}(p,r), c_1, \lambda c_2 \right)$.\\
On the other hand, for $\|u\|_{1,\mathcal{T},0}>1$, it reads
\begin{align*}
\mathcal{J}_{\lambda}(u)
& \geq \left(\frac{1}{r^+}-\frac{1}{s^-}\right)\rho_{\mathcal{T}}(\nabla u)+\left(\frac{1}{r^+}-\frac{1}{s^-}\right) \langle \mathcal{F}^{\prime}(u), u\rangle\nonumber\\
&+\lambda\left(\frac{1}{s^-}-\frac{1}{1-\beta^+}\right)\int_{\Omega}m_2(x)|u|^{1-\beta(x)}dx \nonumber\\
&\geq \left(\frac{1}{r^+}-\frac{1}{s^-}\right)\|u\|_{1,\mathcal{T},0}^{p^{-}}+\left(\frac{1}{r^+}-\frac{1}{s^-}\right) \|u\|_{\mathcal{T}}^{\tau}+c\lambda\left(\frac{1}{s^-}-\frac{1}{1-\beta^+}\right)\|u\|^{1-\beta^-}_{1,\mathcal{T},0},
\end{align*}
which implies that $\mathcal{J}_{\lambda}$ is coercive and bounded below on  $\mathcal{M}$.
\end{proof}

\begin{lemma}\label{Lem:3.3dc}
The functional $\mathcal{J}_{\lambda}$ is continuous on $W_0^{1,\mathcal{T}}(\Omega)$.
\end{lemma}
\begin{proof}
We know that $\varrho_{\mathcal{T}}$ is of class $C^{1}(W_0^{1,\mathcal{T}}(\Omega), \mathbb{R})$.
Let's assume that $u_n \to u$ in $W_0^{1,\mathcal{T}}(\Omega)$. Then
\begin{align}\label{e3.5acb}
&|\mathcal{J}_{\lambda}(u_n )-\mathcal{J}_{\lambda}(u)|\leq |\varrho_{\mathcal{T}}(u_n )-\varrho_{\mathcal{T}}(u)|\nonumber\\
&+C_0\int_{\Omega}\left(||u_n |^{p(x)}-|u|^{p(x)}|+||u_n |^{q(x)}-|u|^{q(x)}|+||u_n |^{r(x)}-|u|^{r(x)}|\right)dx\nonumber\\
&+\frac{1}{s^-}\int_{\Omega}m_1(x)||u_n |^{s(x)}-|u|^{s(x)}|dx+\frac{\lambda}{1-\beta^+}\int_{\Omega}m_2(x)||u_n |^{1-\beta(x)}-|u|^{1-\beta(x)}|dx,
\end{align}
where $C_0:=\frac{\max\{\|\mu_1\|_{\infty},\|\mu_2\|_{\infty},1\}}{p^-x^{\tau}_{0}}$, $\min_{x_i\in \mathbb{R}^N}|x_i|:=x^{\tau}_{0} \in \mathbb{R}$ and $\min\{|x|^{p(x)},|x|^{q(x)},|x|^{r(x)}\}\geq x_{0}^\tau$.
By the compact embeddings given in Proposition \ref{Prop:2.7a} and in Lemmas \ref{Lem:3.3a}, \ref{Lem:3.3b}, the right-hand side of the inequality in (\ref{e3.5acb}) tends to zero as $n \to \infty$. To see this fact, we only show that $||u_n |^{p(x)}-|u|^{p(x)}| \to 0$ in $L^1(\Omega)$ since the other cases can be obtained in a similar way. Put $h_n(x)=||u_n (x)|^{p(x)}-|u(x)|^{p(x)}|$. Since $u_n \to u$ in $L^{p(x)}(\Omega)$, there exists a subsequence $(u_n)$, not relabelled, and a function $\omega(x)$ in $L^{p(x)}(\Omega)$ such that $u_n(x) \to u(x)$ a.e. in $\Omega$ and $|u_n(x)|\leq \omega(x)$  a.e. in $\Omega$ and for all $n$. Therefore, $\{h_n(x)\} \to 0$  a.e. in $\Omega$ and
\begin{equation}\label{e3.5fa}
\lim_{n \to \infty}\int_{\Omega}|u_n |^{p(x)}=\int_{\Omega}|u|^{p(x)}dx.
\end{equation}
By the Vitali Convergence Theorem (see \cite[Theorem 4.5.4]{bogachev2007measure}), $\{h_n\} \to 0$ in measure in $\Omega$ and the sequence $\{h_n\}$ is uniformly integrable, and hence we have
\begin{equation}\label{e3.5fb}
\lim_{n \to \infty}\int_{\Omega}h_n dx=\lim_{n \to \infty}\int_{\Omega}||u_n |^{p(x)}-|u|^{p(x)}|dx=0.
\end{equation}
Therefore, $\mathcal{J}_{\lambda}$ is continuous on $W_0^{1,\mathcal{T}}(\Omega)$.
\end{proof}

Next, we provide a priori estimate.
\begin{lemma}\label{Lem:3.3bd}
If $(u_{n}) \subset \mathcal{M}$ is a minimizing sequence for $\mathcal{J}_{\lambda}$, that is, $\mathcal{J}_{\lambda}(u_{n})\to \inf_{\mathcal{M}}\mathcal{J}_{\lambda}$ as $n \to \infty$, then there exists a real number $\delta>0$ such that
$$
\delta \leq \|u_{n}\|_{1,\mathcal{T},0}.
$$
\end{lemma}
\begin{proof}
We give a proof by contradiction. To do so, assume that there exists a minimizing $(u_{n}) \subset \mathcal{M}$ such that $u_{n}\rightarrow 0$ in $W_0^{1,\mathcal{T}}(\Omega)$. Therefore, $\|u_{n}\|_{1,\mathcal{T},0}<1$ for all $n=1,2,...$. Then, by (\ref{e3.2a}) we have
\begin{align}\label{e3.3aba}
\rho_{\mathcal{T}}(\nabla u)+\langle \mathcal{F}^{\prime}(u_n), u_n\rangle=\int_{\Omega}m_1(x)|u_n|^{s(x)}dx+\lambda\int_{\Omega}m_2(x)|u_n|^{1-\beta(x)} dx.
\end{align}
Using Remark \ref{Rem:3.4ab}, Lemmas \ref{Lem:3.3a}-\ref{Lem:3.3c} and Proposition \ref{Prop:2.7a}, it reads
\begin{align}\label{e3.3dac}
(1+C_{N}(p,r))\|u_n\|_{1,\mathcal{T},0}^{p^-}\geq c_1\|u_n\|_{1,\mathcal{T},0}^{s^+}+c_2\lambda\|u_n\|_{1,\mathcal{T},0}^{1-\beta^-}.
\end{align}
Dividing (\ref{e3.3dac}) by $\|u_n\|_{1,\mathcal{T},0}^{p^-}$ and passing to limit gives
\begin{align}\label{e3.3db}
(1+C_{N}(p,r))\geq \lim_{n \to \infty} \left(c_1\|u_n\|_{1,\mathcal{T},0}^{s^+-p^-}+c_2\lambda\|u_n\|_{1,\mathcal{T},0}^{1-\beta^--p^-}\right).
\end{align}
However this is a contradiction since $p^->1-\beta^-$.
\end{proof}

\begin{lemma}\label{Lem:3.3d} There is a parameter $\lambda^*$ such that for any $\lambda \in (0,\lambda^*)$ the set $\mathcal{M}^0$ is a null set.
\end{lemma}
\begin{proof}
We argue by contradiction, and assume that there exists $u \in \mathcal{M}^0$ with $\|u\|_{1,\mathcal{T},0}>1$. Then using Remark \ref{Rem:3.1} and (\ref{e3.2a}), it reads
\begin{align}\label{e3.5}
0=\langle J^{\prime}(u),u\rangle &= \langle \varrho^{\prime\prime}_{\mathcal{T}}(u),u\rangle +\langle \mathcal{F}^{\prime\prime}(u), u\rangle-\int_{\Omega}m_1(x)s(x)|u|^{s(x)}dx\nonumber\\
&-\lambda\int_{\Omega}m_2(x)(1-\beta(x))|u_n|^{1-\beta(x)} dx\nonumber\\
&\geq p^-\rho_{\mathcal{T}}(\nabla u)+p^-\langle \mathcal{F}^{\prime}(u), u\rangle \nonumber\\
&-s^+\int_{\Omega}m_1(x)|u|^{s(x)}dx-(1-\beta^-)\lambda\int_{\Omega}m_2(x)|u_n|^{1-\beta(x)} dx\nonumber\\
&\geq (p^--(1-\beta^-))\varrho_{\mathcal{T}}(\nabla u)+(p^--(1-\beta^-))\langle \mathcal{F}^{\prime}(u), u\rangle\nonumber\\
&-(s^+-(1-\beta^-))\int_{\Omega}m_1(x)|u|^{s(x)}dx.
\end{align}
Applying Lemmas \ref{Lem:3.3b}, \ref{Lem:3.3c}, and Proposition \ref{Prop:2.7a} provides
\begin{align}\label{e3.5aa}
(s^+-(1-\beta^-))\|u\|^{s^+}_{1,\mathcal{T},0}&\geq (p^--(1-\beta^-))\|u\|^{p^-}_{1,\mathcal{T},0}+(p^--(1-\beta^-))x_{*}^{-\tau}  \|u\|_{\mathcal{H}}^{r^+},
\end{align}
or
\begin{align}\label{e3.5ab}
\|u\|_{1,\mathcal{T},0}\geq \left(\frac{p^--(1-\beta^-)}{s^+-(1-\beta^-)}\right)^{\frac{1}{s^+-p^-}}.
\end{align}
Carrying out a similar analysis gives
\begin{align}\label{e3.8a}
0=\langle J^{\prime}(u),u\rangle & \leq (r^+-s^-)\rho_{\mathcal{T}}(\nabla u)+(r^+-s^-)\langle \mathcal{F}^{\prime}(u), u\rangle \nonumber\\
&+\lambda(s^--(1-\beta^+))\int_{\Omega}m_2(x)|u|^{1-\beta(x)}dx,
\end{align}
or
\begin{align}\label{e3.8ab}
(s^--r^+)\|u\|^{p^-}_{1,\mathcal{T},0}\leq \lambda(s^--(1-\beta^+))\|u\|_{1,\mathcal{T},0}^{1-\beta^-}+(r^+-s^-)x_{*}^{-\tau}\|u\|_{\mathcal{T}}^{r^+}.
\end{align}
Therefore, we conclude that
\begin{align}\label{e3.8ac}
&\|u\|_{1,\mathcal{T},0}\leq \left(\frac{\lambda (s^--(1-\beta^+))}{s^--r^+}\right)^{\frac{1}{p^--(1-\beta^-)}}.
\end{align}
However, when $\lambda$ is sufficiently small, (\ref{e3.5ab}) and (\ref{e3.8ac}) together imply that $\|u\|_{1,\mathcal{T},0}<1$, which is a contradiction.\\
\end{proof}
\begin{lemma}\label{Lem:3.3de}  There is a parameter $\lambda^*$ such that for any $\lambda \in (0,\lambda^*)$ the set $\mathcal{M}$ has at least one element.
\end{lemma}
\begin{proof}
We will show that for any $u \in W_0^{1,\mathcal{T}}(\Omega)\setminus \{0\}$, there exists a unique value $t(u)=t_u>0$ such that $t_u u\in \mathcal{M}$.\\
To this end, let's fix $u \in W_0^{1,\mathcal{T}}(\Omega)\setminus \{0\}$. Then for any $t\geq0$, the scaled functional $\mathcal{J}(tu)$ determines a curve that can be parameterized by
\begin{equation}\label{e3.4de}
   \Phi_u(t) := \mathcal{J}_{\lambda}(tu), \,\,\, t \in [0,\infty).
\end{equation}
Therefore, for $t \in [0,\infty)$, $t_uu\in \mathcal{M}$ if and only if
\begin{equation}\label{e3.5de}
  0=\Phi_u^{\prime}(t)=\frac{d}{dt}\Phi_u(t)\bigg|_{t=t_u}=\langle \mathcal{J}^{\prime}_{\lambda}(t_uu), u \rangle.
\end{equation}
First, we show that $\Phi_u:[0,\infty) \longrightarrow \mathbb{R}$ is continuous. Let's fix $u \in W_0^{1,\mathcal{T}}(\Omega)\setminus \{0\}$ and define the linear functional $\gamma_u:[0,\infty) \longrightarrow W_0^{1,\mathcal{T}}(\Omega)$ by $\gamma_u(t):=tu$. Since $W_0^{1,\mathcal{T}}(\Omega)$ is a vector space, there exists $v_t \in W_0^{1,\mathcal{T}}(\Omega)$ such that $tu=v_t$ a.e. in $\Omega$. Next, let $t_1, t_2 \in [0,\infty)$. Then,
\begin{equation}\label{e3.5df}
\|\gamma_u(t_1)-\gamma_u(t_2)\|_{1,\mathcal{T},0}=\|(t_1-t_2)u\|_{1,\mathcal{T},0}=|t_1-t_2|\|u\|_{1,\mathcal{T},0} \to 0 \text{ as } t_1 \to t_2.
\end{equation}
Thus, $\gamma_u$ is a continuous function. Moreover, by Lemma \ref{Lem:3.3dc}, $\mathcal{J}_{\lambda}$ is continuous on $W_0^{1,\mathcal{T}}(\Omega)$. Hence, if we let $\Phi_u=\mathcal{J}_{\lambda} \circ \gamma_u$, then as a composition function, $\Phi_u$ is continuous.\\
Case I:\\
By Lemmas \ref{Lem:3.3a}-\ref{Lem:3.3c} and Remark \ref{Rem:3.1} it reads
\begin{align}\label{e3.4def}
\Phi_u(t)&=\varrho_{\mathcal{T}}(tu)+\mathcal{F}(tu)
-\int_{\Omega}\frac{m_1(x)|tu|^{s(x)}}{s(x)}dx-\lambda\int_{\Omega}\frac{m_2(x)|tu|^{1-\beta(x)}}{1-\beta(x)}dx \nonumber\\
&\geq \frac{1}{r^{+}} \|tu\|^{r^{+}}_{1,\mathcal{T},0}+\frac{1}{x_{*}^\tau}\|tu\|^{r^+}_{\mathcal{T}}-\frac{1}{s^-}\|tu\|^{s^-}_{1,\mathcal{T},0}-\frac{\lambda}{1-\beta^+} \|tu\|^{1-\beta^+}_{1,\mathcal{T},0}\nonumber\\
&\geq  t^{r^{+}} \left(\frac{1}{r^{+}}-\frac{\lambda}{1-\beta^+} \right) \|u\|^{r^{+}}_{1,\mathcal{T},0} +\frac{t^{\tau}}{x_{*}^\tau}\|u\|^{r^+}_{\mathcal{T}}-\frac{t^{s^-}}{s^-}\|u\|^{s^-}_{1,\mathcal{T},0}.
\end{align}
Therefore, if we let $\lambda \in \left(0,\frac{1-\beta^+}{r^{+}}\right)$, then there exists $M_1>0$ such that when $t \in (0, M_1)$ is small enough, it follows that $\Phi_u(t)>0$.\\
Case II:\\
In a similar way we have
\begin{align}\label{e3.4deg}
\Phi_u(t)&\leq t^{r^+} \left(\frac{1}{p^-}+C_{N}(p,r)\right)\|u\|_{1,\mathcal{T},0}^{r^+}
-\frac{t^{s^+}}{s^+}\int_{\Omega}m_1(x)|u|^{s(x)}dx\nonumber\\
&- \frac{\lambda t^{1-\beta^+}}{1-\beta^-}\int_{\Omega}m_2(x)|u|^{1-\beta(x)}dx.
\end{align}
Therefore, there exists $M_2>0$ such that when $t \in (M_2, \infty)$ is large enough, $\Phi_u(t)<0$.\\
Putting all these together and considering that $\Phi_u(0)=0$ means that $\Phi_u(t)$ assumes a local maximum at $t_u$ in $[0, \infty)$, and hence, it is a critical point for $\Phi_u$. Thus,
\begin{equation}\label{e3.4dde}
  0=\Phi_u^{\prime}(t_u)=\langle \mathcal{J}^{\prime}_{\lambda}(t_uu), u \rangle,
\end{equation}
which implies that  $t_uu\in \mathcal{M}$.\\
\end{proof}

The first main result of the paper is:
\begin{theorem}\label{Thm:3.3a}
There is a parameter $\lambda^*$ such that for any $\lambda \in (0,\lambda^*)$, $\mathcal{J}_{\lambda}$ has a minimizer $m_{\lambda}^+$ on $\mathcal{M}^+$ such that
$\inf_{u\in \mathcal{M}^+}\mathcal{J}_{\lambda}(u)=m_{\lambda}^+<0$, i.e. with the negative energy level, provided $(s^+-p^-)(1-\beta^-)<p^-(s^--r^+)$, and $s^-+\beta^+\leq s^++\beta^-$.
\end{theorem}

\begin{proof}
Let $u \in \mathcal{M}^+$. Then
\begin{align}\label{e3.14}
\mathcal{J}_{\lambda}(u)&\leq \frac{1}{p^-}\rho_{\mathcal{T}}(\nabla u)+\frac{1}{p^-} \langle \mathcal{F}^{\prime}(u), u\rangle-\frac{1}{s^+}\int_{\Omega}m_1(x)|u|^{s(x)}dx\nonumber\\
&-\frac{\lambda}{1-\beta^-}\int_{\Omega}m_2(x)|u|^{1-\beta(x)}dx,
\end{align}
and  from (\ref{e3.2a})
\begin{align}\label{e3.15}
&r^+\rho_{\mathcal{T}}(\nabla u)+r^+ \langle \mathcal{F}^{\prime}(u), u\rangle
-s^-\int_{\Omega}m_1(x)|u|^{s(x)}dx-\lambda(1-\beta^+)\int_{\Omega}m_2(x)|u_n|^{1-\beta(x)} dx>0.
\end{align}
Multiplying (\ref{e3.2a}) by $(-s^-)$ and adding to (\ref{e3.15}) gives
\begin{align}\label{e3.16}
\lambda\int_{\Omega}m_2(x)|u|^{1-\beta(x)}dx\geq \left(\frac{s^--r^+}{s^--(1-\beta^+)}\right)\left(\rho_{\mathcal{T}}(\nabla u)+\langle \mathcal{F}^{\prime}(u), u\rangle\right).
\end{align}
Next, using (\ref{e3.2a}) and (\ref{e3.14}) together, it reads
\begin{align}\label{e3.17}
\mathcal{J}_{\lambda}(u)&\leq \left(\frac{1}{p^-}-\frac{1}{s^+}\right)\left(\rho_{\mathcal{T}}(\nabla u)+\langle \mathcal{F}^{\prime}(u), u\rangle\right)+\lambda\left(\frac{1}{s^+}-\frac{1}{1-\beta^-}\right)\int_{\Omega}m_2(x)|u|^{1-\beta(x)}dx.
\end{align}
Plugging (\ref{e3.16}) in (\ref{e3.17}) and simplifying the expressions gives
\begin{align}\label{e3.18}
\mathcal{J}_{\lambda}(u)
&\leq \left(\left(\frac{1}{p^-}-\frac{1}{s^+}\right)+\left(\frac{1}{s^+}-\frac{1}{1-\beta^-}\right)\left(\frac{s^--r^+}{s^--(1-\beta^+)}\right)\right)\left(\rho_{\mathcal{T}}(\nabla u)+\langle \mathcal{F}^{\prime}(u), u\rangle\right)
\end{align}
However, due to the assumptions, the coefficient of $(\rho_{\mathcal{T}}(\nabla u)+\langle\mathcal{F}^{\prime}(u), u\rangle)$ is negative. Therefore, $\mathcal{J}_{\lambda}(u)<0$ on $\mathcal{M}^+$.\\
Since $\mathcal{J}_{\lambda}$ is coercive and bounded below on $\mathcal{M}^+$, there exists a real number $m_{\lambda}^+$ with $m_{\lambda}^+:=\inf_{u\in \mathcal{M}^+}\mathcal{J}_{\lambda}(u)<0$, and a minimizing sequence $(u^+_n)\subset \mathcal{M}^+$ such that $\lim_{n\rightarrow\infty}\mathcal{J}_{\lambda}(u^+_{n})=\inf_{u\in \mathcal{M}^+}\mathcal{J}_{\lambda}(u)=m_{\lambda}^+<0$. Considering the coercivity of $\mathcal{J}_{\lambda}$,  by the standard arguments we have $u^+_{n}\rightharpoonup u^+$ in  $W_0^{1,\mathcal{T}}(\Omega)$. We claim that $u^+_{n}\to u^+$ in  $W_0^{1,\mathcal{T}}(\Omega)$. To show this, we argue by contradiction, and assume that $u^+_{n} \nrightarrow u^+$  in  $W_0^{1,\mathcal{T}}(\Omega)$.
Using (\ref{e3.2a}), it reads
\begin{align}\label{e3.20}
\mathcal{J}_{\lambda}(u^+_{n})&\geq \left(\frac{1}{r^+}-\frac{1}{s^-}\right)(\rho_{\mathcal{T}}(\nabla u^+_{n})+\langle\mathcal{F}^{\prime}(u^+_{n}), u^+_{n}\rangle)\nonumber\\
&+\lambda\left(\frac{1}{s^-}-\frac{\lambda}{1-\beta^-}\right)\int_{\Omega}m_2(x)|u^+_{n}|^{1-\beta(x)}dx.
\end{align}
Since $u^+_{n}\rightharpoonup u^+$ in  $W_0^{1,\mathcal{T}}(\Omega)$, by Lemma \ref{Lem:3.3a} we have
\begin{align}\label{e3.21}
\int_{\Omega}m_2(x)|u^+|^{1-\beta(x)}dx=\liminf_{n \to \infty}\int_{\Omega}m_2(x)|u^+_{n}|^{1-\beta(x)}dx.
\end{align}
Thus, taking limit in (\ref{e3.20}), and considering that
$$
\rho_{\mathcal{T}}(\nabla u^+)< \liminf_{n \to \infty}\rho_{\mathcal{T}}(\nabla u^+_{n}),
$$
and Fatou's lemma, we get
\begin{align}\label{e3.22}
0>m_{\lambda}^+=\mathcal{J}_{\lambda}(u^+)&\geq \left(\frac{1}{r^+}-\frac{1}{s^-}\right)\|u^+\|^{p^-}_{1,\mathcal{T},0}+\frac{1}{x_{*}^\tau}\left(\frac{1}{r^+}-\frac{1}{s^-}\right)\|u^+\|^{p^-}_{\mathcal{T}} \nonumber\\
&+c\lambda\left(\frac{1}{s^-}-\frac{1}{1-\beta^-}\right)\|u^+\|^{1-\beta^-}_{1,\mathcal{T},0}>0,
\end{align}
if $\|u^+\|_{1,\mathcal{T},0}>1$. This contradiction proves that $u^+_{n}\to u^+$ in  $W_0^{1,\mathcal{T}}(\Omega)$. In conclusion, we obtain that
\begin{align}\label{e3.23}
\mathcal{J}_{\lambda}(u^+)=\lim_{n\rightarrow\infty}\mathcal{J}_{\lambda}(u^+_{n})=\inf_{u\in \mathcal{M}^+}\mathcal{J}_{\lambda}(u)=m_{\lambda}^+<0.
\end{align}
\end{proof}

The second main result of the paper is:
\begin{theorem}\label{Thm:3.3b}
There is a parameter $\lambda^*$ such that for any $\lambda \in (0,\lambda^*)$, $\mathcal{J}_{\lambda}$ has a minimizer $m_{\lambda}^-$ on $\mathcal{M}^-$ such that $0<m_{\lambda}^-=\inf_{u\in \mathcal{M}^-}\mathcal{J}_{\lambda}(u)$, i.e. with the positive energy level.

\end{theorem}
\begin{proof}
Let $u \in \mathcal{M}$. Without loss of generality, we may assume that $\|u\|_{1,\mathcal{T},0}>1$ since the other case leads to the same result. Employing (\ref{e3.2a}), and using Lemmas \ref{Lem:3.3a} and \ref{Lem:3.3c} and Remark \ref{Rem:3.4ab} provides
\begin{align}\label{e3.24}
\mathcal{J}_{\lambda}(u)&\geq\left(\left(\frac{1}{r^+}-\frac{1}{s^-}\right)+\lambda\left(\frac{1}{s^-}-\frac{1}{1-\beta^+}\right)\right)\|u\|^{1-\beta^-}_{1,\mathcal{H},0}
+\frac{1}{x_{*}^\tau}\left(\frac{1}{r^+}-\frac{1}{s^-}\right)\|u\|^{p^-}_{\mathcal{H}}.
\end{align}
If we let $\lambda<\frac{(s^--r^+)(1-\beta^+)}{r^+(s^--(1-\beta^+))}$ in (\ref{e3.24}), we get $\mathcal{J}_{\lambda}(u)>0$. However, since sets $\mathcal{M}^+$ and $\mathcal{M}^-$ are disjoint, and $\mathcal{J}_{\lambda}(u)<0$ on $\mathcal{M}^+$, we must have $u \in \mathcal{M}^-$.\\
Considering the fact that $\mathcal{J}_{\lambda}$ is coercive and bounded below on $\mathcal{M}^-$, we can find a real number $m_{\lambda}^-$ with $m_{\lambda}^-=\inf_{u\in \mathcal{M}^-}\mathcal{J}_{\lambda}(u)>0$. This verifies the existence of a minimizing sequence $(u^-_n)\subset \mathcal{M}^-$ such that $\lim_{n\rightarrow\infty}\mathcal{J}_{\lambda}(u^-_{n})=\inf_{u\in \mathcal{M}^-}\mathcal{J}_{\lambda}(u)=m_{\lambda}^->0$. By the coercivity of $\mathcal{J}_{\lambda}$ and the reflexivity of $W_0^{1,\mathcal{T}}(\Omega)$, $u^-_{n}\rightharpoonup u^-$ (weakly) in  $W_0^{1,\mathcal{T}}(\Omega)$. \\

Note that if $u^- \in \mathcal{M}^-$, then there exists a constant $t_0>0$ such that $t_0u^- \in \mathcal{M}^-$ and $\mathcal{J}_{\lambda}(u^-)\geq \mathcal{J}_{\lambda}(t_0u^-)$. To see this, we use (\ref{e3.2bc})
\begin{align}\label{e3.25}
\langle J^{\prime}(t_0u^-), t_0u^-\rangle &= \langle \varrho^{\prime\prime}_{\mathcal{T}}(t_0u^-),t_0u^-\rangle +\langle \mathcal{F}^{\prime\prime}(t_0u^-), t_0u^-\rangle-\int_{\Omega}m_1(x)s(x)|t_0u^-|^{s(x)}dx\nonumber\\
&-\lambda\int_{\Omega}m_2(x)(1-\beta(x))|t_0u^-|^{1-\beta(x)} dx \nonumber\\
&\leq r^+t_0^{r^{+}}\left(\rho_{\mathcal{T}}(\nabla u^-)+r^+\langle \mathcal{F}^{\prime}(u^-), u^-\rangle\right) -t_0^{s^{-}}\int_{\Omega}m_1(x)s(x)|u^-|^{s(x)}dx\nonumber\\
&-t_0^{1-\beta^{+}}\lambda\int_{\Omega}m_2(x)(1-\beta(x))|u^-|^{1-\beta(x)} dx<0
\end{align}
which means that $t_0u^- \in \mathcal{M}^-$ for $t_0>1$ (the other case follows the same way since $0<1-\beta^{+}<s^-$).\\
Next, we shall show that $u^-_{n}\to u^-$ in  $W_0^{1,\mathcal{T}}(\Omega)$. We argue by contradiction, and assume that $u^-_{n} \nrightarrow u^-$  in  $W_0^{1,\mathcal{T}}(\Omega)$. Therefore,
\begin{align}\label{e3.26}
\mathcal{J}_{\lambda}(t_0u^-)&\leq
\frac{t_0^{r^+}}{p^-}\left(\rho_{\mathcal{T}}(\nabla u^-)+\langle \mathcal{F}^{\prime}(u^-), u^-\rangle\right)\nonumber\\
&-\frac{t_0^{s^-}}{s^+}\int_{\Omega}m_1(x)|u^-|^{s(x)}dx-\frac{\lambda t_0^{1-\beta^+}}{1-\beta^-}\int_{\Omega}m_2(x)|u^-|^{1-\beta(x)}dx.
\end{align}
Using the compact embeddings in Lemmas \ref{Lem:3.3a}-\ref{Lem:3.3b}, and Fatou's lemma we obtain
\begin{align}\label{e3.27}
\mathcal{J}_{\lambda}(t_0u^-)&\leq
\frac{t_0^{r^+}}{p^-}\left(\rho_{\mathcal{T}}(\nabla u^-)+\langle \mathcal{F}^{\prime}(u^-), u^-\rangle\right)\nonumber\\
&-\frac{t_0^{s^-}}{s^+}\int_{\Omega}m_1(x)|u^-_{n}|^{s(x)}dx-\frac{\lambda t_0^{1-\beta^+}}{1-\beta^-}\int_{\Omega}m_2(x)|u^-_{n}|^{1-\beta(x)}dx\nonumber\\
&\leq
\lim_{n \to \infty}\left[\frac{t_0^{r^+}}{p^-}\left(\rho_{\mathcal{T}}(\nabla u^-_{n})+\langle \mathcal{F}^{\prime}(u^-_{n}), u^-_{n}\rangle\right)\right.\nonumber\\
&\left.-\frac{t_0^{s^-}}{s^+}\int_{\Omega}m_1(x)|u^-_{n}|^{s(x)}dx-\frac{\lambda t_0^{1-\beta^+}}{1-\beta^-}\int_{\Omega}m_2(x)|u^-_{n}|^{1-\beta(x)}dx\right]\nonumber\\
&< \lim_{n \to \infty}\mathcal{J}_{\lambda}(t_0u_n^-)\leq \lim_{n \to \infty}\mathcal{J}_{\lambda}(u_n^-)=m_{\lambda}^-.
\end{align}
However, (\ref{e3.27}) implies a contradiction due to the definition of $m_{\lambda}^-$. Therefore,
\begin{align}\label{e3.28}
\mathcal{J}_{\lambda}(u^-)=\lim_{n\rightarrow\infty}\mathcal{J}_{\lambda}(u^-_{n})=\inf_{u\in \mathcal{M}^-}\mathcal{J}_{\lambda}(u)=m_{\lambda}^->0.
\end{align}
\end{proof}

\begin{corollary}\label{Cor:4.1}
Note that since $\mathcal{J}_{\lambda}(u^\pm)=\mathcal{J}_{\lambda}(|u^\pm|)$, we may assume $u^\pm \geq 0$. However, by Lemma \ref{Lem:3.3bd}, we must have $u^\pm > 0$.  Moreover,  by definition of $\mathcal{M}^-$ and $\mathcal{M}^+$,  $|u^\pm| \in \mathcal{M}^\pm$, and hence, $u^\pm > 0$ are two distinct positive weak solutions to problem (\ref{e1.1}).
\end{corollary}

\section{Application}
In this section, we provide an example to illustrate the main results of the paper.
Assume the following:
\begin{itemize}
\item [$\bullet$] $\Omega=B(0,1) \subset \mathbb{R}^3$ (a ball of radius 1).
\item [$\bullet$] $p(x)=2+\frac{|x|}{3}$, $q(x)=2.5+\frac{|x|}{3}$, $r(x)=3+\frac{|x|}{3}$, $s(x)=4.8+\sin(\pi|x|^2)$, $\beta(x)=0.5+0.4|x|$.
\item [$\bullet$] $m_1(x)=\chi_{B(0,1/2)}(x)$, $m_2(x)=e^{-|x|^2}$, $\mu_1(x)=\frac{1}{1+|x|}$, $\mu_2(x)=\frac{1}{2+|x|}$,  $\lambda>0$.
\end{itemize}
Then problem (\ref{e1.1}) turns into
\begin{equation}\label{e4.1}
\begin{cases}
\begin{array}{rlll}
&-\mathrm{div}\left(|\nabla u|^{\frac{|x|}{3}}\nabla u+\frac{1}{1+|x|}|\nabla u|^{0.5+\frac{|x|}{3}}\nabla u+\frac{1}{2+|x|}|\nabla u|^{1+\frac{|x|}{3}}\nabla u\right)\\
&+\frac{|u|^{\frac{|x|}{3}}u}{|x|^{2+\frac{|x|}{3}}}+\frac{1}{1+|x|}\frac{|u|^{0.5+\frac{|x|}{3}}u}{|x|^{2.5+\frac{|x|}{3}}}+\frac{1}{2+|x|}\frac{|u|^{1+\frac{|x|}{3}}u}{|x|^{3+\frac{|x|}{3}}}\\
&=\chi_{B(0,1/2)}(x)|u|^{2.8+\sin(\pi|x|^2)}u+\lambda e^{-|x|^2}{u^{-(0.5+0.4|x|)}} \quad \text{ in } B(0,1),\\
&u>0  \qquad \quad \quad \quad \quad \quad \quad \quad \quad \quad \quad \quad \quad \quad \quad \quad \quad \quad \quad \text{ in } B(0,1),\\
&u=0  \qquad \quad \quad \quad \quad \quad \quad \quad \quad \quad \quad \quad \quad \quad \quad \quad \quad \quad \quad  \text{ on } B(0,1). \tag{$\mathcal{P_\lambda}_{*}$}
\end{array}
\end{cases}
\end{equation}
This model could describe fluid flow in a heterogeneous porous medium, more specifically, groundwater flow in heterogeneous media, where:
\newline
\begin{itemize}
\item [$\bullet$] $u(x)$: hydraulic head (pressure),
\item [$\bullet$] $\nabla u(x)$: hydraulic gradient, which determines the direction and speed of groundwater flow,
\item [$\bullet$] $\chi_{B(0,1/2)}(x)|u|^{2.8+\sin(\pi|x|^2)}u$: nonlinear reactions in localized zones (e.g., areas with chemical contaminants),
\item [$\bullet$] $\lambda e^{-|x|^2}{u^{-(0.5+0.4|x|)}}$: singular sinks modeling wells or fractures with high extraction rates,
\item [$\bullet$]$\mu_1(x)=\frac{1}{1+|x|}$, $\mu_2(x)=\frac{1}{2+|x|}$: heterogeneity factors accounting for permeability variations,
\item [$\bullet$] $\frac{|u|^{\frac{|x|}{3}}u}{|x|^{2+\frac{|x|}{3}}}+\frac{1}{1+|x|}\frac{|u|^{0.5+\frac{|x|}{3}}u}{|x|^{2.5+\frac{|x|}{3}}}+\frac{1}{2+|x|}\frac{|u|^{1.5+\frac{|x|}{3}}u}{|x|^{3+\frac{|x|}{3}}}$: diffusion impeded by singularities (e.g., well bores or fractures).
\end{itemize}
It can be easily derived from the functions assumed above that $p^-=2$, $p^+=2+1/3$, $q^-=2.5$, $q^+=2.5+1/3$, $r^-=3$, $r^-=3+1/3$ $s^-=4.8$, $s^+=5.8$, $\beta^-=0.5$, $\beta^+=0.9$, and $s^+<p^*(x)$ since $6<\frac{Np(x)}{N-p(x)}<\frac{69}{7}$. Thus conditions $(A_{\beta})$, $(H_{1})$ and $(H_{2})$ hold. If we let $\alpha(x)=\gamma(x)=|x|^2+1$, $(A_{\alpha})$ and $(A_{\gamma})$ hold, too. Lastly, conditions $(s^+-p^-)(1-\beta^-)<p^-(s^--r^+)$ and $s^-+\beta^+\leq s^++\beta^-$ are satisfied as well. \newline In conclusion, by Theorems \ref{Thm:3.3a}-\ref{Thm:3.3b} and Corollary \ref{Cor:4.1}, there is a parameter $\lambda_*$ such that for any $\lambda \in (0,\lambda_*)$, problem (\ref{e4.1}) has at least two positive solutions on $\mathcal{M}$.

\section*{Conflict of Interest}
The author declared that he has no conflict of interest.

\section*{Data Availability}
No data is used to conduct this research.

\section*{Funding}
This work was supported by Athabasca University Research Incentive Account [140111 RIA].

\section*{ORCID}
https://orcid.org/0000-0002-6001-627X

\bibliographystyle{tfnlm}
\bibliography{references}

\end{document}